\documentclass[11pt]{article}

\usepackage{amsmath,amssymb,amsthm}
\usepackage{mathrsfs}
\usepackage{hyperref}
\usepackage{tikz-cd}
\usepackage{graphicx} 
\theoremstyle{definition}
\newtheorem{definition}{Definition}[section]

\theoremstyle{plain}
\newtheorem{theorem}[definition]{Theorem}
\newtheorem{proposition}[definition]{Proposition}
\newtheorem{lemma}[definition]{Lemma}
\newtheorem{corollary}[definition]{Corollary}

\newtheorem{remark}[definition]{Remark}

\numberwithin{equation}{section}

\title{Hopf Images of Hopf algebra Coactions }

\author{Arnab Bhattacharjee}
\date{}

\begin{document}

\maketitle

\begingroup
\renewcommand{\thefootnote}{}
\footnotetext{%
\newline
The author acknowledges support from HORIZON-MSCA-2021-SE-01-CaLIGOLA.
}
\endgroup

\begin{abstract}
We introduce the Hopf image of a right coaction of a Hopf algebra on an
algebra as the smallest Hopf subalgebra through which the coaction
factors. We establish its basic properties and show that the induced
coaction on the Hopf image is inner-faithful. For coactions on Nichols
algebras arising from Yetter--Drinfeld modules, we prove that the Hopf
image is the smallest Hopf subalgebra detected by the underlying
coaction on the Yetter--Drinfeld module. We give explicit examples of proper non-trivial Hopf images
arising from bosonizations.
\end{abstract}

\tableofcontents

\section{Introduction}

Effectiveness is a fundamental aspect of symmetry. In the classical
setting, an action of a group on an object may have a non-trivial kernel,
and the effective symmetry is obtained by removing the part of the group
which acts trivially. For a quantum symmetry described by a Hopf algebra
coaction, the analogous question is to determine which part of the
ambient Hopf algebra is actually detected by the coaction.

Let $A$ be an algebra and let
\[
\delta:A\longrightarrow A\otimes H
\]
be a right coaction of a Hopf algebra $H$. We ask whether there exists a
canonical smallest Hopf subalgebra $H_\delta\subseteq H$ through which
$\delta$ factors, and whether the resulting coaction of $H_\delta$ is
effective in the appropriate quantum sense. This leads naturally to the
notions of \emph{Hopf image} and \emph{inner-faithful coaction}.

The notion of Hopf image was introduced by Banica and Bichon
\cite{banica2010hopf} in the context of representations of Hopf
algebras. Their construction associates to a representation of a minimal
Hopf quotient encoding the quantum symmetry detected by the
representation. For coactions, the natural counterpart is a Hopf
subalgebra of the ambient Hopf algebra. In this paper we develop this
construction and define the Hopf image $H_\delta$ of a coaction to be
the smallest Hopf subalgebra of $H$ satisfying
\[
\delta(A)\subseteq A\otimes H_\delta.
\]

Our first goal is to establish the basic theory of Hopf images of
coactions. We prove that the Hopf image is characterized by a universal
property and that the induced coaction
\[
\delta_{\mathrm{im}}:A\longrightarrow A\otimes H_\delta
\]
is inner-faithful. Thus every coaction admits a canonical factorization
through an inner-faithful coaction, and $H_\delta$ provides the effective
quantum symmetry detected by $\delta$. We further describe the Hopf
image in terms of the coefficients of the coaction and establish its
invariance under isomorphisms of comodule algebras. We also study its
behaviour under natural constructions of coactions, in particular
tensor products.

Our second goal is to produce explicit non-trivial examples of proper
Hopf images. For a right-right Yetter--Drinfeld module
$V\in\mathcal{YD}^{H}_{H}$, the canonical coaction on its Nichols
algebra $\mathcal{B}(V)$ provides a natural class of examples. We prove
that the Hopf image of this coaction is precisely the smallest Hopf
subalgebra $K_V\subseteq H$ through which the coaction on $V$ factors:
\[
H_{\rho_{\mathcal{B}(V)}}=K_V.
\]
This reduces the construction of proper Hopf images on Nichols algebras
to the determination of the coefficient Hopf subalgebras of
Yetter--Drinfeld modules.

We conclude with explicit examples, including a bosonization
\[
H=\mathcal{B}(W)\#K
\]
for which a suitable Yetter--Drinfeld module $V$ gives
\[
\mathbb{C}1\subsetneq
H_{\rho_{\mathcal{B}(V)}}\subsetneq H.
\]
In particular, this shows that a genuinely non-commutative and
non-cocommutative ambient Hopf algebra can co-act on a Nichols algebra
through a strictly smaller, non-trivial Hopf algebra.

\medskip
\noindent
\textbf{Structure of the paper.}
Section~2 recalls the necessary
background on Hopf algebras and coactions of Hopf algebras. Section~3 develops the theory of
Hopf images and inner-faithful coactions, including their universal
property and basic structural properties including tensor products of Hopf images. Section~4 is devoted to
examples of Hopf images and inner faithful coactions arising from quantized coordinate algebras, Nichols algebras and bosonizations.
\medskip
\noindent

\section{Preliminaries}

Throughout the paper, $H$ denotes a Hopf algebra with coproduct $\Delta$,
counit $\varepsilon$, and a bijective antipode $S$.
An algebra $A$ is always assumed to be associative and unital. In this section we follow literature from text books \cite{ brzezinski2003corings,majid2002quantum}.

\begin{definition}
A \emph{right coaction} of a Hopf algebra $H$ on an algebra $A$ is an algebra morphism
\[
\delta : A \longrightarrow A \otimes H
\]
satisfying
\[
(\delta \otimes \mathrm{id}_{H})\circ \delta
= (\mathrm{id}_{A} \otimes \Delta)\circ \delta,
\qquad
(\mathrm{id}_{A} \otimes \varepsilon)\circ \delta = \mathrm{id}_{A}.
\]
\end{definition}

We use Sweedler notation $\delta(a)= a_{(0)}\otimes a_{(1)}$.
\section{Hopf image of a coaction}
The notion of Hopf image, introduced by Banica and Bichon
\cite{banica2010hopf} for representations of Hopf algebras, provides a
canonical way to isolate the effective quantum symmetry of an action.
In the present section, we develop an analogous construction for
coactions of Hopf algebras on algebras. In the coaction setting, the
Hopf image is realized as a Hopf subalgebra of the ambient Hopf algebra:
we define it as a universal factorization object encoding the effective
quantum symmetry detected by the coaction.

\subsection{Factorizations of coactions}

\begin{definition}\label{def:facto}
Let $\delta:A\to A\otimes H$ be a right coaction of a Hopf algebra $H$ on an algebra $A$. A \emph{factorization} of $\delta$ is a triple $(L,\iota_L,\delta_L)$ where
$L\subseteq H$ is a Hopf subalgebra with inclusion $\iota_L:L\hookrightarrow H$
and $\delta_L:A\to A\otimes L$ is a right $L$-coaction such that
\[
\delta=(\mathrm{id}\otimes\iota_L)\circ\delta_L.
\]
\end{definition}

\subsection{Hopf image of a coaction}

\begin{definition}\label{def:Hopfimage}
Let $\delta:A\to A\otimes H$ be a right coaction of a Hopf algebra $H$ on an algebra $A$. The \emph{Hopf image} of $\delta$, denoted $H_\delta$, is the initial object in
the category of factorizations of $\delta$.
\end{definition}

\begin{theorem}\label{thm:Hopfimage}
Let $H$ be a Hopf algebra, $A$ an algebra, and
\[
\delta : A \longrightarrow A \otimes H
\]
a right coaction. Then there exists a Hopf subalgebra
\[
H_\delta \subseteq H
\]
and a right $H_\delta$-coaction
\[
\delta_{\mathrm{im}} : A \longrightarrow A \otimes H_\delta
\]
such that:
\begin{enumerate}
\item $\delta = (\mathrm{id}_A \otimes \iota_{H_{\delta}})\circ \delta_{\mathrm{im}}$, where
      $\iota_{H_{\delta}} : H_\delta \hookrightarrow H$ is the inclusion;
\item $(H_\delta,\iota_{H_{\delta}},\delta_{\mathrm{im}})$ is the initial object in the category of factorizations of $\delta$.
\end{enumerate}
Equivalently, $H_\delta$ is the smallest Hopf subalgebra $L \subseteq H$ such that
\[
\delta(A) \subseteq A \otimes L .
\]
\end{theorem}
\begin{proof}
Let $\mathcal{S}$ denote the collection of all Hopf subalgebras
$L \subseteq H$ satisfying $\delta(A) \subseteq A \otimes L$.
Define
\[
H_\delta := \bigcap_{L \in \mathcal{S}} L .
\]

Since intersections of subalgebras (resp.\ subcoalgebras) are again
subalgebras (resp.\ subcoalgebras), and the antipode preserves
intersections, $H_\delta$ is a Hopf subalgebra of $H$.

By construction, $\delta(A) \subseteq A \otimes L$ for every
$L \in \mathcal{S}$, hence
\[
\delta(A) \subseteq A \otimes \bigcap_{L \in \mathcal{S}} L
= A \otimes H_\delta .
\]
Thus $\delta$ uniquely factors as
\[
\delta = (\mathrm{id}_A \otimes \iota_{H_{\delta}})\circ \delta_{\mathrm{im}},
\]
where $\delta_{\mathrm{im}}$ is the restriction of $\delta$ to
$A \otimes H_\delta$.

Now let $(L,\iota_{L},\delta_L)$ be any factorization of $\delta$.
By definition, $L \in \mathcal{S}$, hence $H_\delta \subseteq L$.
The inclusion $f : H_\delta \hookrightarrow L$ is the unique Hopf algebra
morphism satisfying
\[
\delta_L = (\mathrm{id}_A \otimes f)\circ \delta_{\mathrm{im}}
\quad\text{and}\quad
(\mathrm{id}_{A}\otimes \iota_{L})\circ (\mathrm{id}_{A}\otimes f)= \mathrm{id}_{A}\otimes \iota_{H_{\delta}}
\]
Therefore $(H_\delta,\iota_{H_{\delta}},\delta_{\mathrm{im}})$ is initial among
all factorizations of $\delta$.
\end{proof}

\begin{proposition}\label{prop:generator}
Let $H$ be a Hopf algebra, $A$ an algebra, and
\[
\delta : A \longrightarrow A \otimes H
\]
a right $H$-coaction. Then the Hopf image $H_\delta$ of $\delta$ coincides with
the Hopf subalgebra of $H$ generated by the subspace
\[
\mathcal{C}_\delta
:=
\operatorname{span}_{\mathbb{C}}
\left\{
(\omega\otimes\operatorname{id})\delta(a)
\mid
a\in A,\ \omega\in A^*
\right\}.
\]
\end{proposition}

\begin{proof}
Let $K \subseteq H$ be the Hopf subalgebra generated by the subspace
$\mathcal{C}_\delta$.
For each $a \in A$, write $\delta(a)=\sum a_{(0)}\otimes a_{(1)}$.
Then $a_{(1)}\in\mathcal{C}_\delta\subseteq K$, hence
$\delta(a)\in A\otimes K$ for all $a\in A$.
Thus $\delta$ factors through $K$, and therefore $H_\delta\subseteq K$.

Since $\delta(A)\subseteq A\otimes H_\delta$, applying
$(\omega\otimes\mathrm{id})$ shows that
$\mathcal{C}_\delta\subseteq H_\delta$.
Because $H_\delta$ is a Hopf subalgebra, it contains the Hopf subalgebra
generated by $\mathcal{C}_\delta$, namely $K$.
Therefore $K\subseteq H_\delta$.
Combining the two inclusions yields $H_\delta=K$.
\end{proof}

\begin{lemma}\label{lem:minimality}
Let $H$ be a Hopf algebra, $A$ an algebra, and
\[
\delta : A \longrightarrow A \otimes H
\]
a right $H$-coaction. Let $(L,\iota_L,\delta_L)$ be any factorization of $\delta$, that is,
$L \subseteq H$ is a Hopf subalgebra with inclusion $\iota_L : L \hookrightarrow H$ and
$\delta_L : A \to A \otimes L$ is a right $L$-coaction such that
\[
\delta = (\mathrm{id}_A \otimes \iota_L)\circ \delta_L .
\]
Then the Hopf image $H_\delta$ of $\delta$ satisfies
\[
H_\delta \subseteq L .
\]
\end{lemma}

\begin{proof}
By definition, the Hopf image $H_\delta$ is the smallest Hopf subalgebra
$K \subseteq H$ such that $\delta(A) \subseteq A \otimes K$.
Since $(L,\iota_L,\delta_L)$ is a factorization of $\delta$, we have
$\delta(A) \subseteq A \otimes L$, hence $L$ is one of the Hopf subalgebras
appearing in the defining intersection for $H_\delta$.
Therefore $H_\delta \subseteq L$, as claimed.
\end{proof}

\subsection{Inner-faithful coactions}

\begin{definition}\label{def:innerfaithful}
Let $H$ be a Hopf algebra, $A$ an algebra, and
\[
\delta : A \longrightarrow A \otimes H
\]
a right $H$-coaction. Let $H_\delta \subseteq H$ denote the Hopf image of $\delta$.
The coaction $\delta$ is called \emph{inner faithful} if
\[
H_\delta = H .
\]
\end{definition}

\begin{proposition}\label{prop:innerfaithful}
Let $H$ be a Hopf algebra, $A$ an algebra, and
\[
\delta : A \longrightarrow A \otimes H
\]
a right $H$-coaction. Let $H_\delta \subseteq H$ denote the Hopf image of $\delta$.
Then the following statements are equivalent:
\begin{enumerate}
\item $\delta$ is inner faithful, i.e.\ $H_\delta = H$;
\item There exists no proper Hopf subalgebra $L \subsetneq H$ such that
      $\delta(A) \subseteq A \otimes L$;
\item For every factorization $(L,\iota_L,\delta_L)$ of $\delta$, the inclusion
      $\iota_L : L \hookrightarrow H$ is an isomorphism.
\end{enumerate}
\end{proposition}

\begin{proof}
$(1) \Rightarrow (2)$.
If $\delta(A) \subseteq A \otimes L$ for some Hopf subalgebra $L \subseteq H$,
then $(L,\iota_L,\delta_L)$ is a factorization of $\delta$.
By minimality of the Hopf image, $H_\delta \subseteq L$.
If $L$ were proper, this would contradict $H_\delta = H$.
Hence no such proper Hopf subalgebra exists.

$(2) \Rightarrow (3)$.
Let $(L,\iota_L,\delta_L)$ be any factorization of $\delta$.
Then $\delta(A) \subseteq A \otimes L$, so by assumption $L = H$.
Therefore $\iota_L$ is an isomorphism.

$(3) \Rightarrow (1)$.
By definition, $(H_\delta,\iota_{H_{\delta}},\delta_{\mathrm{im}})$ is a factorization
of $\delta$. Applying $(3)$ to this factorization shows that
$\iota_{H_{\delta}} : H_\delta \hookrightarrow H$ is an isomorphism, hence
$H_\delta = H$.
\end{proof}

\begin{proposition}\label{prop:universitality}
Let $H$ be a Hopf algebra, $A$ an algebra, and
\[
\delta : A \longrightarrow A \otimes H
\]
a right $H$-coaction. Let $(H_\delta,\iota_{H_\delta},\delta_{\mathrm{im}})$ denote
the Hopf image of $\delta$, where
$\iota_{H_\delta} : H_\delta \hookrightarrow H$ is the inclusion and
$\delta_{\mathrm{im}} : A \to A \otimes H_\delta$ is the induced coaction. Then, the coaction $\delta_{\mathrm{im}}$ is inner faithful, and conversely, if $(L,\iota_L,\delta_L)$ is a factorization of $\delta$ such that
      the coaction $\delta_L : A \to A \otimes L$ is inner faithful, then
      $L \cong H_\delta$ as Hopf algebras.
\end{proposition}

\begin{proof}
By definition, $(H_\delta,\iota_{H_\delta},\delta_{\mathrm{im}})$ is a factorization
of $\delta$ and $H_\delta$ is the smallest Hopf subalgebra of $H$ such that
$\delta(A) \subseteq A \otimes H_\delta$.
Suppose that $\delta_{\mathrm{im}}$ factors through a proper Hopf subalgebra
$K \subsetneq H_\delta$, i.e.\ $\delta_{\mathrm{im}}(A) \subseteq A \otimes K$.
Then $(K,\iota_K,\delta_K)$ would be a factorization of $\delta$, contradicting
the minimality of $H_\delta$. Hence $\delta_{\mathrm{im}}$ is inner faithful.

Conversely, let $(L,\iota_L,\delta_L)$ be a factorization of $\delta$ such that $\delta_L$
is inner faithful. By minimality of the Hopf image, we have
$H_\delta \subseteq L$.
On the other hand, since $\delta = (\mathrm{id}\otimes\iota_L)\circ\delta_L$,
the image of $\delta_L$ is contained in $A \otimes L$, and by definition of the
Hopf image we also have $L_\delta \subseteq H_\delta$ for the coaction $\delta_L$.
Inner faithfulness of $\delta_L$ implies $L = L_\delta$.
Therefore $L \cong H_\delta$ as Hopf algebras.
\end{proof}
\begin{remark}
The Hopf image thus provides the unique factorization of a coaction with
inner-faithful structure Hopf algebra. In particular, every coaction admits a
canonical reduction to an inner-faithful one by passing to its Hopf image.
\end{remark}
\begin{proposition}
Let $H$ be a Hopf algebra and
\[
\delta_A : A \longrightarrow A \otimes H
\]
a right $H$-coaction on an algebra $A$. Let $\theta : A \to B$ be an algebra
isomorphism, and define a right $H$-coaction on $B$ by
\[
\delta_B := (\theta \otimes \mathrm{id}) \circ \delta_A \circ \theta^{-1}
: B \longrightarrow B \otimes H .
\]
Then the Hopf images of $\delta_A$ and $\delta_B$ are isomorphic as Hopf
algebras. More precisely,
\[
H_{\delta_A} = H_{\delta_B} \subseteq H .
\]
\end{proposition}
\begin{proof}
Recall that the Hopf image of a coaction is the smallest Hopf subalgebra of $H$
through which the coaction factors. Equivalently, it is the Hopf subalgebra
generated by the coefficients of the coaction.

Let $a \in A$ and write
\[
\delta_A(a) =  a_{(0)} \otimes a_{(1)} .
\]
For $b = \theta(a) \in B$, we have
\[
\delta_B(b)
= (\theta \otimes \mathrm{id})\delta_A(a)
=  \theta(a_{(0)}) \otimes a_{(1)} .
\]
Thus the coefficients appearing in $\delta_B(b)$ are exactly the same elements
$a_{(1)} \in H$ as those appearing in $\delta_A(a)$.

Since $\theta$ is an algebra isomorphism, every element of $B$ is of the form
$\theta(a)$ for some $a \in A$, and therefore the set of coefficients of
$\delta_B$ coincides with the set of coefficients of $\delta_A$.
Consequently, the Hopf subalgebra of $H$ generated by the coefficients of
$\delta_B$ is equal to the Hopf subalgebra generated by the coefficients of
$\delta_A$.

By the coefficient description of the Hopf image, it follows that
\[
H_{\delta_B} = H_{\delta_A},
\]
which proves the claim.
\end{proof}
\begin{remark}
The Hopf image of a coaction is therefore invariant under isomorphisms
of comodule algebras. In particular, the effective quantum symmetry
encoded by the Hopf image is an invariant of the coaction up to
isomorphism.
\end{remark}
The Hopf image also behaves naturally with respect to constructions on
coactions. We first consider tensor products of comodule algebras. Given
right coactions of $H_1$ on $A$ and $H_2$ on $B$, there is a natural
right coaction of $H_1\otimes H_2$ on $A\otimes B$. The following
proposition describes the relation between the Hopf image of this
tensor product coaction and the Hopf images of the original coactions.
\begin{proposition}
    Let $H_{1}, H_{2}$ be Hopf algebras and $A,B$ be right $H_1$ and $H_2$ comodule algebras respectively. Let $\delta_{1}: A\rightarrow A\otimes H_1$ and $\delta_{2}: B\rightarrow B\otimes H_{2}$ be right coactions respectively. Then
    \begin{enumerate}
        \item $\delta_{3}:= (\mathrm{id}_{A}\otimes \tau\otimes \mathrm{id}_{H_{2}})\circ (\delta_{1}\otimes \delta_{2})$ is a right coaction of $H_{1}\otimes H_{2}$ on $A\otimes B$, where $\tau: H_{1}\otimes B\rightarrow B\otimes H_{1}$ is a flip map;

        \item There exists an injective Hopf algebra map from $(H_{1}\otimes H_{2})_{\delta_{3}}$ to $H_{1,\delta_{1}}\otimes H_{2, \delta_{2}}$.
    \end{enumerate}
\end{proposition}

\begin{proof}
    \begin{enumerate}
        \item In order to show that $\delta_{3}$ is a right coaction of $H_{1}\otimes H_{2}$ on $A\otimes B$, it is sufficient to show the following identities:
        \begin{enumerate}
            \item $(\mathrm{id}_{A\otimes B}\otimes \Delta_{H_{1}\otimes H_{2}})\circ \delta_{3}= (\delta_{3}\otimes \mathrm{id}_{H_{1}\otimes H_{2}})\circ\delta_{3}$;
            \item $(\mathrm{id}_{A\otimes B}\otimes \epsilon_{H_{1}\otimes H_{2}})\circ \delta_{3}= \mathrm{id}_{A\otimes B}$
        \end{enumerate}
        Where, $\Delta_{H_{1}\otimes H_{2}}$ and $\epsilon_{H_{1}\otimes H_{2}}$ are coproduct and counit repectively on $H_{1}\otimes H_{2}$. Moreover, $\Delta_{H_{1}\otimes H_{2}}$ is given by $(\mathrm{id}_{H_{1}}\otimes T\otimes \mathrm{id}_{H_{2}})\circ (\Delta_{H_{1}}\otimes \Delta_{H_{2}})$, where $T$ is a flip map between $H_{1}\otimes H_{2}$ and $H_{2}\otimes H_{1}$, and $\epsilon_{H_{1}\otimes H_{2}}$ is given by $\epsilon_{H_{1}}\otimes \epsilon_{H_{2}}$.\\
        \\
        We consider $a\otimes b\in A\otimes B$, and we have the following,
        \begin{align*}
            (\mathrm{id}_{A\otimes B}\otimes \Delta_{H_{1}\otimes H_{2}})\circ \delta_{3}(a\otimes b)&= (\mathrm{id}_{A\otimes B}\otimes \Delta_{H_{1}\otimes H_{2}})(\mathrm{id}_{A}\otimes \tau\otimes \mathrm{id}_{H_{2}})\circ (\delta_{1}\otimes \delta_{2})(a\otimes b)\\
            &= (\mathrm{id}_{A\otimes B}\otimes \Delta_{H_{1}\otimes H_{2}})(\mathrm{id}_{A}\otimes \tau\otimes \mathrm{id}_{H_{2}})\circ (\delta_{1}a\otimes \delta_{2}b)\\
            &= (\mathrm{id}_{A\otimes B}\otimes \Delta_{H_{1}\otimes H_{2}})(\mathrm{id}_{A}\otimes \tau\otimes \mathrm{id}_{H_{2}})(a_{(0)}\otimes a_{(1)}\otimes b_{(0)}\otimes b_{(1)})\\
            &= (\mathrm{id}_{A\otimes B}\otimes \Delta_{H_{1}\otimes H_{2}})(a_{(0)}\otimes b_{(0)}\otimes a_{(1)}\otimes b_{(1)})\\
            &= a_{(0)}\otimes b_{(0)}\otimes \Delta_{H_{1}\otimes H_{2}}(a_{(1)}\otimes b_{(1)})\\
            &= a_{(0)}\otimes b_{(0)}\otimes a_{(1)}\otimes b_{(1)}\otimes a_{(2)}\otimes b_{(2)}
        \end{align*}
        Next, \begin{align*}
            (\delta_{3}\otimes \mathrm{id}_{H_{1}\otimes H_{2}})\circ \delta_{3}(a\otimes b)&=  (\delta_{3}\otimes \mathrm{id}_{H_{1}\otimes H_{2}})(\mathrm{id}_{A}\otimes \tau\otimes \mathrm{id}_{H_{2}})\circ (\delta_{1}\otimes \delta_{2})(a\otimes b)\\
            &= (\delta_{3}\otimes \mathrm{id}_{H_{1}\otimes H_{2}})(a_{(0)}\otimes b_{(0)}\otimes a_{(1)}\otimes b_{(1)})\\
            &= \delta_{3}(a_{(0)}\otimes b_{(0)})\otimes a_{(1)}\otimes b_{(1)}\\
            &= a_{(0)}\otimes b_{(0)}\otimes a_{(1)}\otimes b_{(1)}\otimes a_{(2)}\otimes b_{(2)}
        \end{align*}
    
    Therefore, we have the identity $(a)$ i.e. $(\mathrm{id}_{A\otimes B}\otimes \Delta_{H_{1}\otimes H_{2}})\circ \delta_{3}= (\delta_{3}\otimes \mathrm{id}_{H_{1}\otimes H_{2}})\circ\delta_{3}$.
    Next, to show the identity $(b)$, we have the following
    \begin{align*}
        (\mathrm{id}_{A\otimes B}\otimes \epsilon_{H_{1}\otimes H_{2}})\circ \delta_{3}(a\otimes b)&= (\mathrm{id}_{A\otimes B}\otimes \epsilon_{H_{1}\otimes H_{2}})(\mathrm{id}_{A}\otimes \tau\otimes \mathrm{id}_{H_{2}})\circ (\delta_{1}\otimes \delta_{2})(a\otimes b)\\
        &= \mathrm{id}_{A\otimes B}\otimes \epsilon_{H_{1}\otimes H_{2}})(a_{(0)}\otimes b_{(0)}\otimes a_{(1)}\otimes b_{(1)})\\
        &= a_{(0)}\otimes b_{(0)}\otimes \epsilon_{H_{1}\otimes H_{2}}(a_{(1)}\otimes b_{(1)})\\
        &= a_{(0)}\otimes b_{(0)}\otimes \epsilon_{H_{1}}(a_{(1)})\otimes \epsilon_{H_{2}}(b_{(1)})\\
        &= a_{(0)}\otimes b_{(0)}\\
        &= a\otimes b
    \end{align*}
    \item 
We have $(H_{1,\delta_{1}}, \iota_{H_{1,\delta_{1}}}, \delta_{1,\mathrm{im}})$ and $(H_{2,\delta_{2}}, \iota_{H_{2,\delta_{2}}}, \delta_{2,\mathrm{im}})$ be Hopf images of the right coactions $\delta_{1}$ and $\delta_{2}$ respectively. Therefore, we have the following diagrams

\begin{tikzcd}
A \arrow[r, "\delta_{1}"] \arrow[rd, "{\delta_{1,\mathrm{im}}}"'] & A\otimes H_{1}                                                                               &  & B \arrow[rd, "{\delta_{2,\mathrm{im}}}"'] \arrow[r, "\delta_{2}"]                                                                                                     & B\otimes H_{2}                                                                              \\
                                                                  & {A\otimes H_{1,\delta_{1}}} \arrow[u, "{\mathrm{id}_{A}\otimes \iota_{H_{1,\delta_{1}}} }"'] &  &                                                                                                                                                                       & {B\otimes H_{2,\delta_{2}}} \arrow[u, "{\mathrm{id}_{B}\otimes \iota_{H_{2,\delta_{2}}}}"'] \\
                                                                  & A\otimes B \arrow[rr, "\delta_{3}"] \arrow[rrd, "\overline{\delta_{3}}"']                    &  & A\otimes B\otimes H_{1}\otimes H_{2}                                                                                                                                  &                                                                                             \\
                                                                  &                                                                                              &  & {A\otimes B\otimes H_{1,\delta_{1}}\otimes H_{2,\delta_{2}}} \arrow[u, "{\mathrm{id}_{A\otimes B}\otimes \iota_{H_{1,\delta_{1}}}\otimes \iota_{H_{2,\delta_{2}}}}"'] &                                                                                            
\end{tikzcd}
Here, $\overline{\delta_{3}}:= (\mathrm{id}_{A}\otimes \tau_{\mathrm{im}}\otimes \mathrm{id}_{H_{2}, \delta_{2}})\circ (\delta_{1,\mathrm{im}}\otimes \delta_{2,\mathrm{im}})$ and $\tau_{\mathrm{im}}: H_{1,\delta_{1}}\otimes B\rightarrow B\otimes H_{1, \delta_{1}}$ is a flip map. Since, by $(1)$, $\delta_{3}$ is a right coaction of $H_{1}\otimes H_{2}$, so, $(H_{1, \delta_{1}}\otimes H_{2, \delta_{2}}, \iota_{H_{1,\delta_{1}}}\otimes \iota_{H_{2,\delta_{2}}}, \overline{\delta_{3}})$ is a factorization of the coaction $\delta_{3}$. Therefore, by minimality of the Hopf image of $H_{1}\otimes H_{2}$, we have $(H_{1}\otimes H_{2})_{\delta_{3}}\subseteq H_{1,\delta_{1}}\otimes H_{2,\delta_{2}}$. Hence, there exists an injective Hopf algebra morphism from $(H_{1}\otimes H_{2})_{\delta_{3}}$ to $H_{1,\delta_{1}}\otimes H_{2,\delta_{2}}$.
    \end{enumerate}
\end{proof}

\section{Examples of Hopf image and inner-faithful coactions}
In this section, we give examples of inner-faithful coactions. One of the concrete example is for the case of the quantized coordinate algebra of a complex semisimple Lie group. We also give an example where a non-trivial Hopf image of a coaction can be obtained from Nichols algebras.
\medskip

\subsection{Examples from coproduct of a Hopf algebra}

We begin this section with the simplest form of example, and therefore we consider the coproduct of a Hopf algebra.

\begin{proposition}\label{prop:regular-inner-faithful}
Let $H$ be a Hopf algebra.
Then the right coaction given by the coproduct
\[
\delta := \Delta : H \longrightarrow H \otimes H
\]
is inner--faithful.
\end{proposition}

\begin{proof}
Let $K \subseteq H$ be a Hopf subalgebra such that the coaction $\Delta$ factors
through $K$, i.e.
\[
\Delta(H) \subseteq H \otimes K.
\]
Apply $(\varepsilon \otimes \mathrm{id})$ to $\Delta(h)$ for $h \in H$.
Using the counit axiom $(\varepsilon \otimes \mathrm{id})\circ \Delta = \mathrm{id}$,
we obtain
\[
h = (\varepsilon \otimes \mathrm{id})\Delta(h) \in K.
\]
Thus $H \subseteq K$, and hence $K = H$.
Therefore no proper Hopf subalgebra of $H$ can factor the coaction, and $\Delta$
is inner--faithful.
\end{proof}

\begin{remark}
    We consider $L$ be a proper Hopf subalgebra of $H$ such that $\iota: L \hookrightarrow H$ be an inclusion. Then $\delta':= (\mathrm{id}_{L}\otimes \iota)\circ \Delta|_{L}: L \rightarrow L\otimes H$ is a right coaction of $H$ on $L$. In this case, the Hopf-image of $\delta'$ is given by $H_{\delta'}= L\subset H$. This gives us a simplest example where the Hopf image of a coaction is non-trivial.
\end{remark}

\subsection{Examples from coaction induced from quantized coordinate algebra}
Let $G$ be a complex semisimple Lie group and let $L_S$ be a Levi subgroup of $G$. Consider the Hopf algebras $A=\mathcal O_q(G)$ and $H=\mathcal O_q(L_S)$. We have cannonical Hopf algebra surjection map $\pi:\mathcal O_q(G)\twoheadrightarrow \mathcal O_q(L_S)$ (for example, one can see \cite{stokman1999quantized, zwicknagl2009r}), and consider the
right coaction
\[
\delta := (\mathrm{id}\otimes\pi)\circ\Delta_{\mathcal O_q(G)}:
\mathcal O_q(G)\longrightarrow
\mathcal O_q(G)\otimes\mathcal O_q(L_S).
\]

\begin{proposition}\label{prop:levi-inner-faithful}
Let $G$ be a complex semisimple Lie group and $L_S\subset G$ a Levi subgroup.
Then the
right coaction
\[
\delta :
\mathcal O_q(G)\longrightarrow
\mathcal O_q(G)\otimes\mathcal O_q(L_S).
\]
is inner--faithful. 
\end{proposition}

\begin{proof}
Let $K\subseteq \mathcal O_q(L_S)$ be a Hopf subalgebra such that the coaction
$\delta$ factors through $K$, that is,
\[
\delta(\mathcal O_q(G))\subseteq
\mathcal O_q(G)\otimes K.
\]
Applying $(\varepsilon\otimes\mathrm{id})$ to $\delta$, we obtain
\[
\pi(\mathcal O_q(G))\subseteq K.
\]
Since $\pi:\mathcal O_q(G)\twoheadrightarrow \mathcal O_q(L_S)$ is surjective,
it follows that $K=\mathcal O_q(L_S)$.
Hence no proper Hopf subalgebra of $\mathcal O_q(L_S)$ can factor the coaction,
and $\delta$ is inner--faithful.
\end{proof}

\subsection{Examples arising from Nichols algebras}

Let \(H\) be a Hopf algebra and let $V\in \mathcal{YD}^{H}_{H}$ be a right-right Yetter--Drinfeld module. Denote by $\mathcal{B}(V)$ the Nichols algebra of \(V\) (one can refer to \cite{heckenberger2020hopf,majid2002quantum} for literature on the Yetter--Drinfeld modules and Nichols algebras). The right \(H\)-coaction on \(V\) extends
to a right \(H\)-coaction on \(\mathcal{B}(V)\),
\[
\rho_{\mathcal{B}(V)}:
\mathcal{B}(V)\longrightarrow
\mathcal{B}(V)\otimes H,
\]
which is an algebra homomorphism. Since $\mathcal{B}(V)$ is generated as an algebra by $V$, one expects
the Hopf image of this coaction to be determined by the $H$-comodule
structure of $V$. We make this precise in the following proposition.
We consider the smallest Hopf subalgebra
\(K_V\subseteq H\) such that
\[
\rho_V(V)\subseteq V\otimes K_V,
\]
where
\[
\rho_V:V\longrightarrow V\otimes H
\]
denotes the given Yetter--Drinfeld coaction. The following results shows
that \(K_V\) is precisely the Hopf image of the induced coaction on the
Nichols algebra. In order to do so, we shall first prove a general result relevant to this.

\begin{proposition}\label{prop:Nichols}
Let \(H\) be a Hopf algebra and let \(A\) be a right \(H\)-comodule
algebra with coaction
\[
\rho_A:A\longrightarrow A\otimes H.
\]
Let \(V\subseteq A\) be a right \(H\)-subcomodule which generates
\(A\) as an algebra. Thus, writing
\[
\rho_V:=\rho_A|_V,
\]
we have
\[
\rho_V:V\longrightarrow V\otimes H.
\]

Let \(K_V\) denote the smallest Hopf subalgebra of \(H\) such that
\[
\rho_V(V)\subseteq V\otimes K_V.
\]
Then
\[
H_{\rho_A}=K_V.
\]
In particular, the Hopf image of the coaction on \(A\) is completely
determined by the \(H\)-comodule structure of the algebra generators
\(V\).
\end{proposition}

\begin{proof}
We first show that
\[
H_{\rho_A}\subseteq K_V.
\]

By definition of \(K_V\),
\[
\rho_V(V)\subseteq V\otimes K_V.
\]
Since \(K_V\) is a subalgebra of \(H\), the tensor product
\(A\otimes K_V\) is a subalgebra of \(A\otimes H\). Moreover,
\(\rho_A\) is an algebra homomorphism and \(A\) is generated as an
algebra by \(V\). Hence
\[
\rho_A(A)\subseteq A\otimes K_V.
\]
Indeed, if
\[
a=v_1v_2\cdots v_n,
\qquad v_i\in V,
\]
then
\[
\rho_A(a)
=
\rho_A(v_1)\rho_A(v_2)\cdots\rho_A(v_n)
=
\rho_V(v_1)\rho_V(v_2)\cdots\rho_V(v_n)
\in A\otimes K_V.
\]
Thus \(\rho_A\) factors through the Hopf subalgebra \(K_V\). By the
minimality of the Hopf image \(H_{\rho_A}\), it follows that
\[
H_{\rho_A}\subseteq K_V.
\]

Conversely, let
\[
L:=H_{\rho_A}.
\]
By the definition of the Hopf image,
\[
\rho_A(A)\subseteq A\otimes L.
\]
Since \(V\subseteq A\) and \(\rho_V=\rho_A|_V\), we obtain
\[
\rho_V(V)\subseteq V\otimes L.
\]
Thus \(L\) is a Hopf subalgebra of \(H\) through which the coaction
\(\rho_V\) factors. By the defining minimality of \(K_V\), we therefore
have
\[
K_V\subseteq L=H_{\rho_A}.
\]

Combining the two inclusions gives
\[
H_{\rho_A}=K_V.
\]
\end{proof}

\begin{corollary} \label{cor:Nichols}
Let \(H\) be a Hopf algebra and let $V\in\mathcal{YD}^{H}_{H}$ be a right-right Yetter--Drinfeld module. Let $\mathcal{B}(V)$ be the Nichols algebra of \(V\), equipped with its canonical right
\(H\)-coaction
\[
\rho_{\mathcal{B}(V)}:
\mathcal{B}(V)\longrightarrow
\mathcal{B}(V)\otimes H
\]
extending the Yetter--Drinfeld coaction
\[
\rho_V:V\longrightarrow V\otimes H.
\]
Let \(K_V\) be the smallest Hopf subalgebra of \(H\) satisfying
\[
\rho_V(V)\subseteq V\otimes K_V.
\]
Then
\[
H_{\rho_{\mathcal{B}(V)}}=K_V.
\]
\end{corollary}

\begin{proof}
By definition, the Nichols algebra \(\mathcal{B}(V)\) is generated as
an algebra by its degree-one component \(V\). Moreover, \(V\) is a
right \(H\)-subcomodule of \(\mathcal{B}(V)\), and the canonical
\(H\)-coaction on \(\mathcal{B}(V)\) restricts to the given
Yetter--Drinfeld coaction on \(V\). Therefore, the Proposition~\ref{prop:Nichols} applies with $A=\mathcal{B}(V)$. Consequently, $H_{\rho_{\mathcal{B}(V)}}=K_V$.
\end{proof}

\begin{remark}
The preceding Corollary~\ref{cor:Nichols} provides a criterion for obtaining proper
non-trivial Hopf images from Nichols algebras. Namely, if \(K_V\) is the
smallest Hopf subalgebra of \(H\) satisfying
\[
\rho_V(V)\subseteq V\otimes K_V,
\]
then
\[
H_{\rho_{\mathcal{B}(V)}}=K_V.
\]
Consequently, if
\[
\mathbb{C}1\subsetneq K_V\subsetneq H,
\]
then the canonical \(H\)-coaction on \(\mathcal{B}(V)\) has a proper
non-trivial Hopf image:
\[
\mathbb{C}1
\subsetneq
H_{\rho_{\mathcal{B}(V)}}
=
K_V
\subsetneq
H.
\]
Thus, the construction of proper non-trivial Hopf images on Nichols
algebras reduces to the problem of finding Yetter--Drinfeld modules
\(V\in\mathcal{YD}^{H}_{H}\) such that 
\(K_V\) is a proper non-trivial Hopf subalgebra of \(H\).
\end{remark}

\paragraph{A non-trivial example of Hopf image arising from a bosonization}

We now give an example illustrating the preceding Corollary~\ref{cor:Nichols}.
The ambient Hopf algebra in this example is a standard finite-dimensional
bosonization; our point is to use a suitable Yetter--Drinfeld module over
this Hopf algebra to obtain a coaction on a Nichols algebra with a proper
and non-trivial Hopf image.

Let
\[
K=\mathbb{C}(\mathbb{Z}_2\times\mathbb{Z}_2)
=
\mathbb{C}\langle g,h\mid
g^2=h^2=1,\ gh=hg\rangle.
\]
Let $W=\mathbb{C}x$ be the one-dimensional Yetter--Drinfeld module over
$K$ determined by
\[
\rho_W(x)=x\otimes h,
\qquad
x\triangleleft g=x,
\qquad
x\triangleleft h=-x.
\]
The corresponding braiding is given by
\[
c(x\otimes x)=-x\otimes x,
\]
and hence
\[
\mathcal{B}(W)=\mathbb{C}[x]/(x^2).
\]
Consider the bosonization
\[
H=\mathcal{B}(W)\# K.
\]
We identify $K$ with its canonical Hopf subalgebra of $H$ and write
$x$ for $x\#1$. Thus $H$ is generated by $x,g,h$ with relations
\[
x^2=0,\qquad
g^2=h^2=1,\qquad
gh=hg,\qquad
gx=xg,\qquad
hx=-xh.
\]
In particular, $g$ is a central group-like element of $H$. The resulting Hopf algebra (bosonization) is eight-dimensional and is a standard
pointed Hopf algebra. In fact, after the change of generators
\[
g'=h,\qquad h'=gh,
\]
its presentation agrees with the Hopf algebra \(A_{2,2}\) in the
classification of pointed Hopf algebras of dimension eight
\cite[\S~4]{beattie2013classifying}.

We now consider the following right-right Yetter--Drinfeld module over
$H$. Let $V=\mathbb{C}v$ and define
\[
v\triangleleft g=-v,
\qquad
v\triangleleft h=v,
\qquad
v\triangleleft x=0,
\]
together with the right coaction
\[
\rho_V(v)=v\otimes g.
\]
The Yetter--Drinfeld compatibility is immediate on the group-like
generators $g,h$, while for $x$ it follows from
\[
\Delta(x)=x\otimes1+h\otimes x,\qquad
S(x)=-hx,\qquad
v\triangleleft x=0,\qquad x^2=0.
\]
Hence $V\in\mathcal{YD}^{H}_{H}$.

The braiding on $V$ is
\[
c(v\otimes v)
=
v_{(0)}\otimes(v\triangleleft v_{(1)})
=
v\otimes(v\triangleleft g)
=
-v\otimes v.
\]
Consequently,
\[
\mathcal{B}(V)
=
\mathbb{C}[v]/(v^2)
\]
is the exterior algebra on $V$.

We next determine the smallest Hopf subalgebra through which the
coaction on $V$ factors. Since
\[
\rho_V(v)=v\otimes g,
\]
the coefficient Hopf subalgebra is
\[
K_V=\mathbb{C}\langle g\rangle
\cong\mathbb{C}\mathbb{Z}_2.
\]
Indeed, $g$ occurs as a coefficient of $\rho_V(v)$, while
\[
\rho_V(V)\subseteq V\otimes\mathbb{C}\langle g\rangle.
\]
Therefore
\[
\mathbb{C}1
\subsetneq
K_V
=
\mathbb{C}\langle g\rangle
\subsetneq
H.
\]
The preceding Corollary~\ref{cor:Nichols} now yields
\[
\boxed{
H_{\rho_{\mathcal{B}(V)}}
=
K_V
=
\mathbb{C}\langle g\rangle
\subsetneq H.
}
\]

Thus, although the ambient Hopf algebra $H$ is a standard
non-commutative and non-cocommutative bosonization, the canonical
coaction on the exterior Nichols algebra $\mathcal{B}(V)$ detects only
the proper non-trivial Hopf subalgebra
\[
H_{\rho_{\mathcal{B}(V)}}
=
\mathbb{C}\langle g\rangle
\cong\mathbb{C}\mathbb{Z}_2.
\]
This gives, in particular, an explicit example of a proper non-trivial
Hopf image arising from the Nichols-algebra construction.


\section*{Competing Interests}

The author declares that there are no competing interests.

\bibliographystyle{plain}
\bibliography{name}

\bigskip

\noindent
\textsc{Arnab Bhattacharjee}\\
Mathematical Institute\\
Charles University\\
Prague, Czech Republic\\
\textit{E-mail address:} \texttt{bhattacharjeea@karlin.mff.cuni.cz}

\end{document}